\documentclass[10pt,article,reqno]{amsart}
\usepackage{amssymb}
\usepackage{amsthm}
\usepackage{amsmath}
\usepackage{graphicx}
\usepackage{tikz}
\usepackage{tikz-3dplot}

\theoremstyle{plain}
\newtheorem{theorem}{Theorem}
\newtheorem*{maintheorem}{Main Theorem}
\newtheorem{lem}{Lemma}

\def\N{\mathbb{N}}

\def\irA{\mathcal{A}}

\def\im{\mathrm{Im}\,}
\def\tr{\mathrm{Tr}\,}

\begin{document}

\title{Wigner's theorem on Grassmann spaces}

\author{Gy\"orgy P\'al Geh\'er}
\address{Department of Mathematics, University of Reading, Reading RG6 6AX, United Kingdom}
\email{gehergyuri@gmail.com or G.P.Geher@reading.ac.uk}

\begin{abstract}
Wigner's celebrated theorem, which is particularly important in the mathematical foundations of quantum mechanics, states that every bijective transformation on the set of all rank-one projections of a complex Hilbert space which preserves the transition probability is induced by a unitary or an antiunitary operator. 
This vital theorem has been generalised in various ways by several scientists. 
In 2001, Moln\'ar provided a natural generalisation, namely, he provided a characterisation of (not necessarily bijective) maps which act on the Grassmann space of all rank-$n$ projections and leave the system of Jordan principal angles invariant (see \cite{MolnarWigner} and \cite{MolnarWigner2}). 
In this paper we give a very natural joint generalisation of Wigner's and Moln\'ar's theorems, namely, we prove a characterisation of all (not necessarily bijective) transformations on the Grassmann space which fix the quantity $\tr PQ$ (i.e.~the sum of the squares of cosines of principal angles) for every pair of rank-$n$ projections $P$ and $Q$.
\end{abstract}

\maketitle

\section{Introduction and statement of the main result}

Let $H$ be a complex Hilbert space and $I$ stand for the identity operator.
If $n$ is a positive integer, then we denote the set of all rank-$n$ (self-adjoint) projections by $P_n(H)$.
This space can be naturally identified with the \emph{Grassmann space} of all $n$-dimensional subspaces of $H$ using the map $P\mapsto \im P$.
In case when $n=1$, we get the usual projective space that represents the set of all pure states of a quantum system.
For $P,Q \in P_n(H)$ let us call the quantity $\tr PQ$ the \emph{transition probability} between the two projections.
If $n=1$, then this is a commonly used notion in quantum mechanics, furthermore, $\tr PQ = \cos^2 \vartheta$ where $\vartheta$ is the angle between $\im P$ and $\im Q$.
Wigner's theorem characterises symmetry transformations of $P_1(H)$ that respect the transition probability, or equivalently, that leave the angle invariant.
However, this theorem can be significantly improved, namely, we can drop the bijectivity assumption and have a similar conclusion.

\begin{theorem}[E.P. Wigner, see \cite{Wigner}, or \cite{Geher, MolnarWigner, Uhlhorn}]\label{W}
	Let $\phi\colon P_1(H) \to P_1(H)$ be a (not necessarily bijective) transformation which satisfies
		$$
		\tr \phi(P)\phi(Q) = \tr PQ \qquad (P,Q \in P_1(H)).
		$$
	Then $\phi$ is induced by either a linear or a conjugatelinear isometry $V\colon H\to H$, i.e.
		$$
		\phi(P) = VPV^* \qquad (P \in P_1(H)).
		$$
\end{theorem}

The above result is commonly referred to as the \emph{optimal version of Wigner's theorem}.
Various generalisations of this essential result have been provided, we only mention a few of them \cite{BotelhoJamisonMolnar,CMPregi,Chevalier,GSUhlhorn,GyoryUhlhorn,Molnarindef,MolnarWigner2,Molnarindef2,MolnarWigner,Semrlrankone,Semrlquaternion,Semrlalgebra,SemrlUhlhorn}.
This short note is particularly concerned with \emph{Moln\'ar's generalisation} which we explain now.
Assume that $n>1$ and $P, Q \in P_n(H)$, then the \emph{principal angles} between $P$ and $Q$ are the arcuscosines of the $n$ largest singularvalues of $PQ$ (\cite[Exercise VII.1.10]{Bhatia}, \cite[Problem 559]{Kirillov}).
The system of all principal angles is denoted by $\measuredangle(P,Q) := (\vartheta_1,\dots \vartheta_n)$ where $\tfrac{\pi}{2}\geq\vartheta_1\geq \vartheta_2\geq \dots \geq \vartheta_n \geq 0$.
The origin of the notion goes back to Jordan's work \cite{Jordan} and has serious applications, see e.g.~\cite{ARMA, robust, Hotelling, applbook, Li}.
Moln\'ar proved the following.

\begin{theorem}[L. Moln\'ar, \cite{MolnarWigner2,MolnarWigner}]\label{MW}
	Let $\dim H > n \geq 2$ and $\phi\colon P_n(H) \to P_n(H)$ be a (not necessarily bijective) transformation that satisfies
	\begin{equation}\label{princang}
		\measuredangle(\phi(P),\phi(Q)) = \measuredangle(P,Q) \qquad (P,Q \in P_n(H)).
	\end{equation}
	Then either $\phi$ is induced by a linear or a conjugatelinear isometry $V\colon H\to H$, i.e.
	$$
	\phi(P) = VPV^* \qquad (P \in P_n(H)),
	$$
	or we have $\dim H = 2n$ and 
	$$
	\phi(P) = I-VPV^* \qquad (P \in P_n(H)).
	$$
\end{theorem}

As it was revealed in a personal conversation, Moln\'ar's original desire was to prove a more general result.
Namely, note that by the two projections theorem (\cite{BS,Gal,GSUhlhorn}) we have $\tr PQ = \sum_{j=1}^n \cos^2\vartheta_j$, therefore if $\phi$ satisfies \eqref{princang}, then it automatically preservers the transition probability as well (see \eqref{trpres} below).
Actually, in the first few steps of the proof of Theorem \ref{MW} Moln\'ar used only this weaker property, although, there is a point where the methods start to heavily rely on \eqref{princang}.

The aim of the present paper is to provide this missing result which is stated below, and hence giving a very natural joint generalisation of the Wigner and Moln\'ar theorems.

\begin{maintheorem}\label{main}
	Let $\dim H > n \geq 2$ and $\varphi\colon P_n(H) \to P_n(H)$ be a (not necessarily bijective) map which preserves the transition probability, that is
	\begin{equation}\label{trpres}
		\tr \varphi(P)\varphi(Q) = \tr PQ \qquad (P,Q \in P_n(H)).
	\end{equation}
	Then either $\varphi$ is induced by a linear or conjugatelinear isometry $V\colon H\to H$, i.e.
	\begin{equation}\label{regular}
		\varphi(P) = VPV^* \qquad (P \in P_n(H)),
	\end{equation}
	or we have $\dim H = 2n$ and 
	\begin{equation}\label{perpregular}
		\varphi(P) = I-VPV^* \qquad (P \in P_n(H)).
	\end{equation}
\end{maintheorem}

We point out that all the above three theorems hold for real Hilbert spaces as well and their proofs are almost the same, even simpler, as in the complex case.
We present the proof of the Main Theorem in the next section.

Let us note that \eqref{trpres} is equivalent to the following property
	\begin{equation}\label{HSpres}
	\| \varphi(P) - \varphi(Q) \|_{HS} = \| P - Q \|_{HS} \qquad (P,Q \in P_n(H)),
	\end{equation}
where $\|\cdot\|_{HS}$ denotes the Hilbert-Schmidt norm.
Therefore our result describes the general form of not necessarily surjective \emph{isometries of the Grassmannian} with respect to this special norm. 
We mention that recently two papers \cite{BotelhoJamisonMolnar,GSUhlhorn} have been published about the same problem for the case of the operator norm.
However, the characterisation of non-bijective isometries of $P_n(H)$ with respect to the operator norm is still an open problem in the case when $\dim H = \infty$.
We hope that our proof gives some additional insight into that problem as well.


\section{Proof of Main Theorem}

Let $F_s(H)$ be the set of all finite-rank self-adjoint operators on $H$.
We begin with stating a lemma which is a trivial consequence of \cite[Lemma 2.1.5]{Molnarbook} and \cite[Lemma 1]{MolnarWigner}, and which was crucial in \cite{MolnarWigner}, as well as here.

\begin{lem}[L. Moln\'ar]\label{specebb}
	If $\varphi$ satisfies the conditions of Main Theorem, then it has a unique real-linear extension $\Phi\colon F_s(H) \to F_s(H)$ which is injective and satisfies
	\begin{equation}\label{extended_trpres}
		\tr \Phi(A)\Phi(B) = \tr AB \quad (A,B \in F_s(H)).
	\end{equation}
\end{lem}

An immediate consequence of Lemma \ref{specebb} is that if $\dim H < \infty$, then $\Phi$ is a homeomorphism, moreover, by the domain invariance theorem $\varphi$ is a homeomorphism as well.
We call two rank-$n$ projections $P$ and $Q$ \emph{adjacent} if $\dim(\im P \cap \im Q) = n-1$, or equivalently, if $\mathrm{rank}(P-Q) = 2$, and in this case we use the notation $P\overset{a}{\sim} Q$. 
Note that $P\overset{a}{\sim} Q$ implies $P\neq Q$.
It is apparent by the two projections theorem that $P\overset{a}\sim Q$ if and only if $\measuredangle(P,Q)$ contains exactly one non-zero angle.

From now on, we will distinguish two different cases.

\subsection{The $2n$-dimensional case}

Here we will utilise the following special case of Chow's fundamental theorem of geometry of Grassmann spaces.

\begin{theorem}[W.-L. Chow, see \cite{Chow}, or \cite{GSUhlhorn}]\label{Chow}
	Let $\dim H = 2n$ and $\phi\colon P_n(H)\to P_n(H)$ be a continuous bijection which preserves adjacency in both directions, i.e.
	$$
		\phi(P) \overset{a}{\sim} \phi(Q) \iff P \overset{a}{\sim} Q \quad (P,Q\in P_n(H)).
	$$
	Then there exists a linear or conjugatelinear bijection $A\colon H\to H$ such that either
	\begin{equation}\label{regularChow}
		\im \phi(P) = A (\im P) \quad (P\in P_n(H)),
	\end{equation}
	or 
	\begin{equation}\label{perpregularChow}
		\im \phi(P) = \left(A (\im P)\right)^\perp \quad (P\in P_n(H)).
	\end{equation}
\end{theorem}

In the general version of Chow's theorem continuity is not assumed, however, then $A$ can be a non-continuous semilinear bijection as well.
That version also covers the $2n < \dim H < \infty$ case where the conclusion \eqref{perpregularChow} is of course excluded.

Next, we introduce some technical notions.
Let us call $P$ and $Q \in P_n(H)$ \emph{orthogonal adjacent} if $P\overset{a}\sim Q$ and $\vartheta_1 = \tfrac{\pi}{2}$, in notation $P\overset{\perp a}{\sim} Q$.
Similarly, $P,Q \in P_n(H)$ are said to be \emph{non-orthogonal adjacent} if $P\overset{a}\sim Q$ and $\vartheta_1 < \tfrac{\pi}{2}$, in notation $P\overset{\not\perp a}{\sim} Q$.
For any $k\in\N$, subspace $M$, and $P,Q \in P_k(M)$ we define the set
$$
	\irA^{(k)}_{P,Q} = \left\{R\in P_k(M)\colon P+Q-R\in P_k(M)\right\}.
$$
We will show that $\varphi$ preserves non-orthogonal adjacency in both directions in which the following topological characterisation of the relation $\overset{\not\perp a}{\sim}$ plays a crucial role.

\begin{lem}\label{manifold}
	Suppose that $P,Q \in P_n(H)$. Then $\irA^{(n)}_{P,Q}$ is a one-dimensional (real) manifold if and only if $P\overset{\not\perp a}{\sim} Q$.
\end{lem}

\begin{proof}
	Clearly, we have $\irA^{(n)}_{P,P} = \{P\}$, therefore from now on we may assume that $P\neq Q$.
	Let us first investigate the case when $P\overset{a}{\sim} Q$. 
	Then $P$ and $Q$ can be represented by the following block-matrices with respect to the orthogonal decomposition $H = M_1\oplus M_2 \oplus M_3$ where $M_1 = \im P\cap\im Q$, $M_1 \oplus M_2 = \im P + \im Q$, $\dim M_1 = \dim M_3 = n-1$, $\dim M_2 = 2$ and $\mathfrak{p}, \mathfrak{q} \in P_1(M_2)$:
	$$
		P = 
		\left(
			\begin{matrix}
				I_{n-1} & 0 & 0\\
				0 & \mathfrak{p} & 0\\
				0 & 0 & 0_{n-1}
			\end{matrix}
		\right)
		\quad \text{and} \quad 
		Q = 
		\left(
			\begin{matrix}
				I_{n-1} & 0 & 0\\
				0 & \mathfrak{q} & 0\\
				0 & 0 & 0_{n-1}
			\end{matrix}
		\right).
	$$
	Suppose that $R\in \irA^{(n)}_{P,Q}$ and set $S = P+Q-R\in P_n(H)$.
	Since we have
	$$
		\|Rx\|^2 + \|Sx\|^2 = \langle (R+S)x, x\rangle = \langle (P+Q)x, x\rangle = 2\|x\|^2 \quad (x\in M_1)
	$$
	and
	$$
		\|Rx\|^2 + \|Sx\|^2 = \langle (R+S)x, x\rangle = \langle (P+Q)x, x\rangle = 0 \quad (x\in M_3),
	$$
	we immediately infer $M_1 \subseteq \im R \cap \im S$ and $M_3 \subseteq \ker R \cap \ker S$. 
	Thus the block-matrix representations of $R$ and $S$ in the decomposition $H = M_1\oplus M_2 \oplus M_3$ are
	$$
		R = 
		\left(
			\begin{matrix}
				I_{n-1} & 0 & 0\\
				0 & \mathfrak{r} & 0\\
				0 & 0 & 0_{n-1}
			\end{matrix}
		\right)
		\quad \text{and} \quad 
		S = 
		\left(
			\begin{matrix}
				I_{n-1} & 0 & 0\\
				0 & \mathfrak{s} & 0\\
				0 & 0 & 0_{n-1}
			\end{matrix}
		\right),
	$$
	where $\mathfrak{r}, \mathfrak{s} \in P_1(M_2)$, whence we easily conclude the following:
	$$
		\irA^{(n)}_{P,Q} = \left\{ \left(
		\begin{matrix}
		I_{n-1} & 0 & 0\\
		0 & \mathfrak{t} & 0\\
		0 & 0 & 0_{n-1}
		\end{matrix}
		\right) \colon \mathfrak{t} \in \irA^{(1)}_{\mathfrak{p},\mathfrak{q}} \right\}.
	$$
	In particular, $\irA^{(n)}_{P,Q}$ and $\irA^{(1)}_{\mathfrak{p},\mathfrak{q}}$ are homeomorphic.
	
	Next, we investigate the set $\irA^{(1)}_{\mathfrak{p},\mathfrak{q}}$, where $\mathfrak{p} \neq \mathfrak{q}$.
	We shall represent elements of $F_S(M_2)$ by 2$\times$2 Hermitian matrices.
	If $\mathfrak{p} + \mathfrak{q} = I_2$, i.e. $P\overset{\perp a}{\sim}Q$, then obviously $\irA^{(1)}_{\mathfrak{p},\mathfrak{q}} = P_1(M_2)$, hence $\irA^{(n)}_{P,Q}$ is a two-dimensional manifold.
	Suppose that $\mathfrak{p} + \mathfrak{q} \neq I_2$, i.e. $P\overset{\not\perp a}{\sim}Q$, then applying unitary similarity we may assume without loss of generality that 
	$\mathfrak{p} + \mathfrak{q} = 
	\left(
		\begin{matrix}
		s & 0\\
		0 & 2-s
		\end{matrix}
	\right)$ where $0<s<1$.
	Since for any $\mathfrak{r} \in P_1(M_2)$ we have $\tr(\mathfrak{p} + \mathfrak{q} - \mathfrak{r}) = 1$, we infer that $\mathfrak{r} \in \irA^{(1)}_{\mathfrak{p},\mathfrak{q}}$ if and only if $\mathfrak{p} + \mathfrak{q} - \mathfrak{r}$ is singular.
	But this holds exactly when $I_2 - (\mathfrak{p} + \mathfrak{q})^{-1} \mathfrak{r}$ is singular, that is equivalent to $\tr (\mathfrak{p} + \mathfrak{q})^{-1} \mathfrak{r} = 1$, since $(\mathfrak{p} + \mathfrak{q})^{-1} \mathfrak{r}$ is of rank one.
	Therefore an Hermitian $2\times 2$ matrix $A = \left(
	\begin{matrix}
	a_{11} & a_{12}\\
	a_{21} & a_{22}
	\end{matrix}
	\right)$ belongs to $\irA^{(1)}_{\mathfrak{p},\mathfrak{q}}$ if and only if $\tr A = a_{11} + a_{22} = 1$, $\tr (\mathfrak{p} + \mathfrak{q})^{-1} A = \tfrac{a_{11}}{s} + \tfrac{a_{22}}{2-s} = 1$ and $A$ is of rank 1.
	Observe that the two equations immediately yield $a_{11} = \tfrac{s}{2}$ and $a_{22} = \tfrac{2-s}{2}$, and since $A$ has rank one, we also obtain $a_{12} = \overline{a_{21}} = \tfrac{\sqrt{(2-s)s}}{2} e^{it}$ with a real number $t$.
	This implies that $\irA^{(1)}_{\mathfrak{p},\mathfrak{q}}$ is a one-dimensional manifold, and therefore so is $\irA^{(n)}_{P,Q}$.
	
	Finally, let us suppose that $P\neq Q$ and $P \not\overset{a}{\sim} Q$.
	Then there is an orthogonal decomposition $H = H_1\oplus \dots \oplus H_n$ such that $\dim H_j = 2$ for every $j$ and that we have the following block-diagonal representations where $\mathfrak{p}_j, \mathfrak{q}_j \in P_1(M_j)$ $(j = 1,\dots n)$:
	$$
		P = 
		\left(
			\begin{matrix}
				\mathfrak{p}_1 & 0 & \dots & 0\\
				0 & \mathfrak{p}_2 & \dots & 0\\
				\vdots & \vdots & \ddots & \vdots\\
				0 & 0 & \dots & \mathfrak{p}_n
			\end{matrix}
		\right)
		\quad \text{and} \quad 
		Q = 
		\left(
			\begin{matrix}
				\mathfrak{q}_1 & 0 & \dots & 0\\
				0 & \mathfrak{q}_2 & \dots & 0\\
				\vdots & \vdots & \ddots & \vdots\\
				0 & 0 & \dots & \mathfrak{q}_n
			\end{matrix}
		\right).
	$$
	Observe that $\mathfrak{p}_j \neq \mathfrak{q}_j$ holds for at least two indices and that we obviously have
	$$
		\left\{
			\left(
				\begin{matrix}
					\mathfrak{t}_1 & 0 & \dots & 0\\
					0 & \mathfrak{t}_2 & \dots & 0\\
					\vdots & \vdots & \ddots & \vdots\\
					0 & 0 & \dots & \mathfrak{t}_n
				\end{matrix}
			\right)
			\colon
			\mathfrak{t}_j \in \irA^{(1)}_{\mathfrak{p}_j,\mathfrak{q}_j}
		\right\}
		\subset \irA^{(n)}_{P,Q}.
	$$
	Since the left-hand side is a manifold of dimension at least two, the right-hand side cannot be a one-dimensional manifold, which completes the proof.
\end{proof}

Utilising Moln\'ar's lemma we easily obtain the following property:
$$
	\varphi\big(\irA^{(n)}_{P,Q}\big) = \Phi\big(\irA^{(n)}_{P,Q}\big) = \irA^{(n)}_{\Phi(P),\Phi(Q)} = \irA^{(n)}_{\varphi(P),\varphi(Q)} \quad (P,Q \in P_n(H)).
$$
Since $\varphi$ is a homeomorphism, we infer the following equivalence-chain:
$$
	P\overset{\not\perp a}{\sim} Q \iff \irA^{(n)}_{P,Q} \text{ is a one-dimensional manifold}
$$
$$
	\iff \irA^{(n)}_{\varphi(P),\varphi(Q)} \text{ is a one-dimensional manifold } \iff \varphi(P)\overset{\not\perp a}{\sim} \varphi(Q),
$$
i.e.~$\varphi$ preserves non-orthogonal adjacency in both directions.
The lower semicontinuity of the $\mathrm{rank}$ on $F_s(H)$ yields the following for every $P\in P_n(H)$:
$$
\big\{R\in P_n(H) \colon P\overset{\not\perp a}{\sim} R\big\}^- = \{P\}\cup\big\{R\in P_n(H) \colon P\overset{a}{\sim} R\big\}
$$
where $\cdot^-$ denotes the closure. 
Therefore, since $\varphi$ is a homeomorphism, it preserves adjacency in both directions, which implies that $\varphi$ satisfies either \eqref{regularChow} or \eqref{perpregularChow}.
Finally, by \eqref{trpres} the map $A$ preserves orthogonality, and thus $A$ must be a scalar multiple of a unitary or an antiunitary operator which completes the proof of the present case.

Let us make an important observation here.
Clearly, every rank-one projection $\mathfrak{p}\in P_1(H)$ can be expressed as a real-linear combination of $n+1$ rank-$n$ projections 
(\cite[Lemma 2.1.5]{Molnarbook}), moreover, if this linear combination is $\mathfrak{p} = \sum_{j=1}^{n+1} t_j P_j$, then taking the trace of both sides gives $\sum_{j=1}^{n+1} t_j = \tfrac{1}{n}$.
Therefore, in case of \eqref{perpregular} we have
$$
\Phi(\mathfrak{p}) = \sum_{j=1}^{n+1} t_j \Phi(P_j) = V\left(\sum_{j=1}^{n+1} t_j (I_{2n}-P_j)\right)V^* = \tfrac{1}{n}I_{2n} - V\mathfrak{p}V^* \qquad (\mathfrak{p}\in P_1(H)),
$$
and similarly, in case of \eqref{regular} we obtain $\Phi(\mathfrak{p}) = V\mathfrak{p}V^*$ for every $\mathfrak{p}\in P_1(H)$.

\subsection{The general case}
By the following three properties it is apparent that the case of $\dim H < 2n$ follows from the $\dim H > 2n$ case: $P\in P_n(H)$ holds if and only if $I-P \in P_{\dim H - n}(H)$, we have $\tr (I-P)(I-Q) = \dim H - 2n + \tr PQ$ for every $P,Q \in P_n(H)$, and the following map preserves the transition probability:
$$
\psi \colon P_{\dim H - n}(H) \to P_{\dim H - n}(H), \; \psi(\tilde P) = I - \varphi(I-\tilde P) \quad (\tilde P \in P_{\dim H - n}(H)).
$$

Next, assume that $\dim H > 2n$ and fix two orthogonal rank-$n$ projections $P$ and $Q$.
By \eqref{trpres} we obtain that $\varphi(P)$ and $\varphi(Q)$ are orthogonal as well, and since for any $R\in P_n(H)$ we have $R\leq P+Q$ if and only if $R\in\irA^{(n)}_{P,Q}$, we easily conclude $\varphi(R)\leq \varphi(P)+\varphi(Q)$.
By the observation following the $2n$-dimensional case we get either $\Phi(\mathfrak{p}) \in P_1(H)$ ($\mathfrak{p}\in P_1(H)$, $\mathfrak{p}\leq P+Q$), or $\im\Phi(\mathfrak{p}) = \im\varphi(P)\oplus\im\varphi(Q)$  ($\mathfrak{p}\in P_1(H)$, $\mathfrak{p}\leq P+Q$).
Assume for a moment that the second possibility holds. 
If we replace in the above method $Q$ by another $Q'\in P_n(H)$ that is still orthogonal to $P$, then we easily obtain $\im\Phi(\mathfrak{p}) = \im\varphi(P)\oplus\im\varphi(Q')$  ($\mathfrak{p}\in P_1(H)$, $\mathfrak{p}\leq P+Q'$), since we obviously cannot have $\Phi(\mathfrak{p}) \in P_1(H)$ for any $\mathfrak{p}\in P_1(H), \mathfrak{p}\leq P$.
In particular, we obtain $\Phi(P_1(H)) \cap P_1(H) = \emptyset$, whence we infer that $\im\varphi(P)\oplus\im\varphi(Q)$ must be the same subspace for every orthogonal pair $P$ and $Q$, from which we conclude $\im \Phi(A) \subseteq \im\varphi(P)\oplus\im\varphi(Q)$ ($A\in F_s(H)$) that contradicts to the injectivity of $\Phi$.
Therefore we must have $\Phi(P_1(H)) \subset P_1(H)$, and finally, \eqref{extended_trpres}, Theorem \ref{W} and the linearity of $\Phi$ imply \eqref{regular}.


\section*{Acknowledgements}
The author is very grateful to the anonymous referee for several useful suggestions which definitely improved the quality of the paper.
This work was supported by the Engineering and Physical Sciences Research Council grant EP/M024784/1, by the Hungarian National Research, Development and Innovation Office -- NKFIH (grant no.~K115383), and by the "Lend\" ulet" Program (LP2012-46/2012) of the Hungarian Academy of Sciences.


\bibliographystyle{amsplain}

\end{document}